\documentclass[12pt]{extarticle}
\usepackage{extsizes}
\usepackage[reqno]{amsmath} 
\usepackage{amsmath}
\usepackage{amssymb}
\usepackage{amsthm}
\usepackage{amscd}
\usepackage{amsfonts}
\usepackage{graphicx}
\usepackage{dsfont}
\usepackage{fancyhdr}
\usepackage{enumerate}
\usepackage{youngtab}
\usepackage[utf8]{inputenc}
\usepackage{comment}

\usepackage{subfigure}
\usepackage[margin=1.5in]{geometry}
\usepackage{graphicx}
\usepackage{algorithm}
\usepackage{algpseudocode}
\usepackage{color}
\usepackage{hyperref}
\usepackage{notation}

\theoremstyle{theorem}

\newtheorem{corollary}{Corollary}

\newtheorem{lemma}[corollary]{Lemma}
\newtheorem{lemma*}[lem6]{Lemma}

\newtheorem{theorem}[corollary]{Theorem}

\begin{document}

\AtEndDocument{%
  \par
  \medskip
  \begin{tabular}{@{}l@{}}%
    \textsc{Gabriel Coutinho}\\
    \textsc{Dept. of Computer Science} \\ 
    \textsc{Universidade Federal de Minas Gerais, Brazil} \\
    \textit{E-mail address}: \texttt{gabriel@dcc.ufmg.br} \\ \ \\
    \textsc{Emanuel Juliano} \\
    \textsc{Dept. of Computer Science} \\ 
    \textsc{Universidade Federal de Minas Gerais, Brazil} \\
    \textit{E-mail address}: \texttt{emanuelsilva@dcc.ufmg.br}\\ \ \\
    \textsc{Thomás Jung Spier} \\
    \textsc{Dept. of Computer Science} \\ 
    \textsc{Universidade Federal de Minas Gerais, Brazil} \\
    \textit{E-mail address}: \texttt{thomasjung@dcc.ufmg.br}
  \end{tabular}}

\title{Strong cospectrality in trees}
\author{Gabriel Coutinho\footnote{gabriel@dcc.ufmg.br --- remaining affiliations in the end of the manuscript.} \and Emanuel Juliano \and Thomás Jung Spier}
\date{\today}
\maketitle
\vspace{-0.8cm}

\begin{abstract} 
	We prove that no tree contains a set of three vertices which are pairwise strongly cospectral. This answers a question raised by Godsil and Smith in 2017.
\end{abstract}

\begin{center}
\end{center}

\section{Introduction}\label{intro}

Let $G$ be a finite simple graph and $A$ its adjacency matrix. A continuous-time quantum walk can be defined having $G$ as an underlying graph, and in certain models where no external interference exists, all properties of the walk are determined by the spectrum of $A$. A desirable property for a quantum walk is that at certain times the quantum state input at a vertex is transferred to another --- if this occurs with probability $1$, then it is called perfect state transfer, and if it occurs with probability close to $1$, it is called pretty good state transfer.

In both cases, a necessary condition is that the two vertices involved are so that their projections onto the eigenspaces of the graph are either equal or minus each other, in which case the vertices will be called strongly cospectral. Precisely, if $a$ and $b$ are vertices of $G$ and $A = \sum_{r} \theta_r E_r$ is the spectral decomposition of $A$, then $a$ and $b$ are strongly cospectral if $E_r e_a = \pm E_r e_b$, for all $r$, where $e_a$ stands for the characteristic vector of the vertex $a$.

Strongly cospectral vertices have been extensively studied in \cite{godsil2017strongly} and we do not aim to survey all results therein. However, it is enlightening to realize that if two vertices are strongly cospectral, then they are cospectral in conventional usage of the term, meaning, the graphs obtained upon the removal of each are going to have the same spectrum. Cospectral vertices have been studied for a long time, and in the context of trees, they are a key piece in Schwenk's \cite{schwenk1973almost} seminal paper.

Two vertices $a$ and $b$ in $G$ are similar if there is an automorphism of the graph that maps $a$ to $b$. This automorphism, restricted to $G \setminus a$, implies that $G \setminus a \simeq G\setminus b$, and if the latter isomorphism exists even when no automorphism of $G$ maps $a$ to $b$, we say that $a$ and $b$ are pseudo-similar. Again, these concepts have been around for some time, see for instance \cite{kimble1981pseudosimilar,godsil1983graphs}. It is immediate to note that similar and pseudosimilar vertices are cospectral, so it is natural to wonder what is their connection to the concept of strong cospectrality. It is perhaps natural to expect that similar vertices are strongly cospectral but this is false --- for instance, no pair of vertices in $K_n$ with $n \geq 3$ is strongly cospectral. In fact, if $a$ and $b$ are strongly cospectral and there is an automorphism that fixes $a$, then it must also fix $b$. This suggests that strong cospectrality captures a sort of regularity or symmetry that must distinguish the pair of vertices from the remaining vertices of the graph.

Naturally at this point one gets suspicious of the fact that in a given graph there cannot be a set of three or more vertices which are pairwise strongly cospectral. This suspicion is further reinforced by the fact that perfect state transfer, the quantum property inspiring the definition of strong cospectrality, is indeed monogamous (see \cite{kay2011basics}): no three vertices in a given graph can be involved in perfect state transfer with each other. Yet, there are graphs with three or more vertices pairwise strongly cospectral. The easiest examples are vertices of smallest degree in cartesian powers of paths of different lenght, so as long they have simple eigenvalues. With simple eigevalues, cospectrality is equivalent to strong cospectrality (see \cite{godsil2017strongly}), it is enough to simply check that the number of closed walks of fixed length around the vertices is constant for all of them. 

So why bother with trees? First, there is special interest in understanding quantum walks in trees (see for instance \cite{CoutinhoLiu2}) because trees model quantum systems which are likely cheaper and easier to build. Unfortunately, there is no known example of perfect state transfer in tress with more than 3 vertices, and this may as well be a consequence of the fact that strong cospectrality in trees is not as common as it is for other graphs. In fact, our result is the first to display a disparity: even though there are graphs with arbitrarily large sets of pairwise strongly cospectral vertices, no such set will exist in a tree. Second, trees seem to behave differently than graphs when it comes to cospectrality. A famous example is the fact that almost all trees have a cospectral mate \cite{schwenk1973almost}, whereas the opposite is widely believed to be true for graphs in general (a conjecture due to W. Haemers). Our result shows that one further aspect of this difference, and therefore hopefully serves as inspiration for future investigations.

Third, and most importantly, the question on whether there are trees with three or more vertices pairwise strongly cospectral was asked by Godsil and Smith \cite{godsil2017strongly}. We answer it fully in the negative.

In Section \ref{sec:prelim} we introduce the basic facts and notation used throughout the paper. In Section \ref{sec:result}, we state a key lemma and prove our main result. Section \ref{sec:lemma} is dedicated to prove the key lemma.

\section{Graph spectra and polynomials}\label{sec:prelim}

In this paper, we will denote by $\phi^G$ the characteristic polynomial of the graph $G$ in the variable $t$. If $\theta_0,\theta_1,\dots, \theta_d$ are the distinct eigenvalues of the adjacency matrix $A$ of $G$, then we denote by $E_r$ the orthogonal projection onto the  $\theta_r$-eigenspace. 

Two vertices $i$ and $j$ of the graph $G$ are called \textit{cospectral} if $\phi^{G\setminus i}=\phi^{G\setminus j}$. Using walk generating functions, it is possible to write the entries of $E_r$ in terms of these polynomials, as follows:
\begin{equation}
(E_r)_{i,i} = \frac{(t-\theta_r)\phi^{G \setminus i}}{\phi^ G} \bigg|_{t = \theta_r}, \label{eq:diagonal}
\end{equation}

\noindent whereas this is well defined as $\theta_r$ is a pole of order at most $1$ in $\phi^{G\setminus i}/\phi^G$ (see \cite{GodsilAlgebraicCombinatorics,coutinho2021quantum}). From this, it follows that $i$ and $j$ are cospectral  if and only if, for every $r$ in $\{0,1,\dots,d\}$, $(E_r)_{i,i}=(E_r)_{j,j}$.

If, moreover, $E_r e_i = \pm E_r e_j$ for every $r$ in $\{0,1,\dots,d\}$, then $i$ and $j$ are \textit{strongly cospectral}. It is easy to verify that this is equivalent to requiring that they are cospectral and that $(E_r)_{i,i} = \pm (E_r)_{i,j}$ for every $r$ in $\{0,1,\dots,d\}$. Walk generating functions also provide the following expression
\begin{equation}(E_r)_{i,j} = \frac{(t-\theta_r)\sqrt{\phi^{G \setminus i}\phi^{G \setminus j} - \phi^G \phi^{G \setminus \{i,j\}}}}{\phi^ G} \bigg|_{t = \theta_r} \label{eq:offdiagonal}
\end{equation}
and so we obtain the following result:

\begin{theorem}[Corollary 8.4 in \cite{godsil2017strongly}]\label{thm:strcospec}
Vertices $i$ and $j$ of a graph $G$ are strongly cospectral if and only if $\phi^{G \setminus i} = \phi^{G \setminus j}$ and all poles of $\phi^{G\setminus \{i,j\}}/\phi^G$ are simple.
\end{theorem}

Note that as a consequence, if $i$ and $j$ are in distinct connected components of the graph $G$, then $i$ and $j$ are not strongly cospectral.

At first sight, the expression within the square root in \eqref{eq:offdiagonal} is not clearly a perfect square, but in fact it is, and given by the following expression:

\begin{lemma}[Lemma 2.1 in \cite{GodsilAlgebraicCombinatorics}]\label{lem:wronskian} 
Let $i$ and $j$ be vertices in the graph $G$. Then,
\[\phi^{G\setminus i}\phi^{G\setminus j}-\phi^{G\setminus\{i,j\}}\phi^G=\left(\sum_{P:i\to j}\phi^{G\setminus P}\right)^2,\]
where the sum is over all the paths from $i$ to $j$.
\end{lemma}

Our proof will require manipulating ratios of characteristic polynomials of a graph and its vertex deleted subgraphs, and for that end, we have found it more convenient to introduce the following notation. Given a graph $G$ and vertex $i$, have
\begin{equation}
\alpha_i^G =\dfrac{\phi^G}{\phi^{G\setminus i}}.
\label{eq:alpha}
\end{equation}

We end this section establishing a description of the graph of this rational function.

\begin{lemma}[Theorem 1.5 in \cite{GodsilAlgebraicCombinatorics} ]\label{lem:derivative}
Let $G$ be a graph. Then, the derivative of $\phi^G$ is given by $(\phi^G)'=\sum_{i\in V(G)}\phi^{G\setminus i}$.
\end{lemma}

\begin{lemma}\label{lemma:derivative_quotient} Let $i$ be a vertex in the graph $G$. Then, $(\alpha_i^G)'(t)\geq 1$ for every $t$ that is not a pole of $\alpha_i^G$. In particular, $\alpha_i^G(t)$ has only simple zeros and poles, and is increasing and surjective on each of its branches.
\end{lemma}

\begin{proof} Naturally, all zeros and poles of $\alpha_i^G$ are real. By taking the derivative in \eqref{eq:alpha} and by Lemmas \ref{lem:derivative} and \ref{lem:wronskian}, it follows that
	
\[(\alpha_i^G)'= 1+\sum_{j\in V(G\setminus i)}\left(\dfrac{\sum_{P:i\to j}\phi^{G\setminus P}}{\phi^{G\setminus i}}\right)^2.\]

This implies that $(\alpha_i^G)'(t)\geq 1$ for every $t$ that is not a zero of $\phi^{G\setminus i}$. It follows by continuity that $(\alpha_i^G)'(t)\geq 1$ for every $t$ that is not a pole of $\alpha_i^G$. As a consequence, $\alpha_i^G$ is increasing and surjective in each of its branches and all of its zeros are simple.

Since $\deg(\phi^G)=\deg(\phi^{G\setminus i})+1$, the number of zeros of $\alpha_i^G$ is one more than the number of poles counted with multiplicity of $\alpha_i^G$. But in each branch, because $\alpha_i^G$ is increasing, there can only be one zero of $\alpha_i^G$. Putting this all together, it follows that all the poles of $\alpha_i^G$ are also simple.
\end{proof}

\section{Main result}\label{sec:result}

The result we prove in this section is that no tree has three (or more) pairwise strongly cospectral vertices. First, we show that if such a set of pairwise strongly cospectral vertices exist, then there must be a cut vertex whose removal separates all three of them in different connected components. In fact, note that it is enough to assume they are pairwise cospectral.

\begin{lemma}\label{lem:notinpath}
If three vertices in a tree are pairwise cospectral, then they do not lie on a path.
\end{lemma}
\begin{proof}
Let $T$ be a tree as in Figure~\ref{fig:same_path}, where the vertices $i$, $j$ and $k$ are on the same path. We will prove that $j$ cannot be cospectral to both $i$ and $k$. 

\begin{figure}
    \centering
    \includegraphics[scale=0.2]{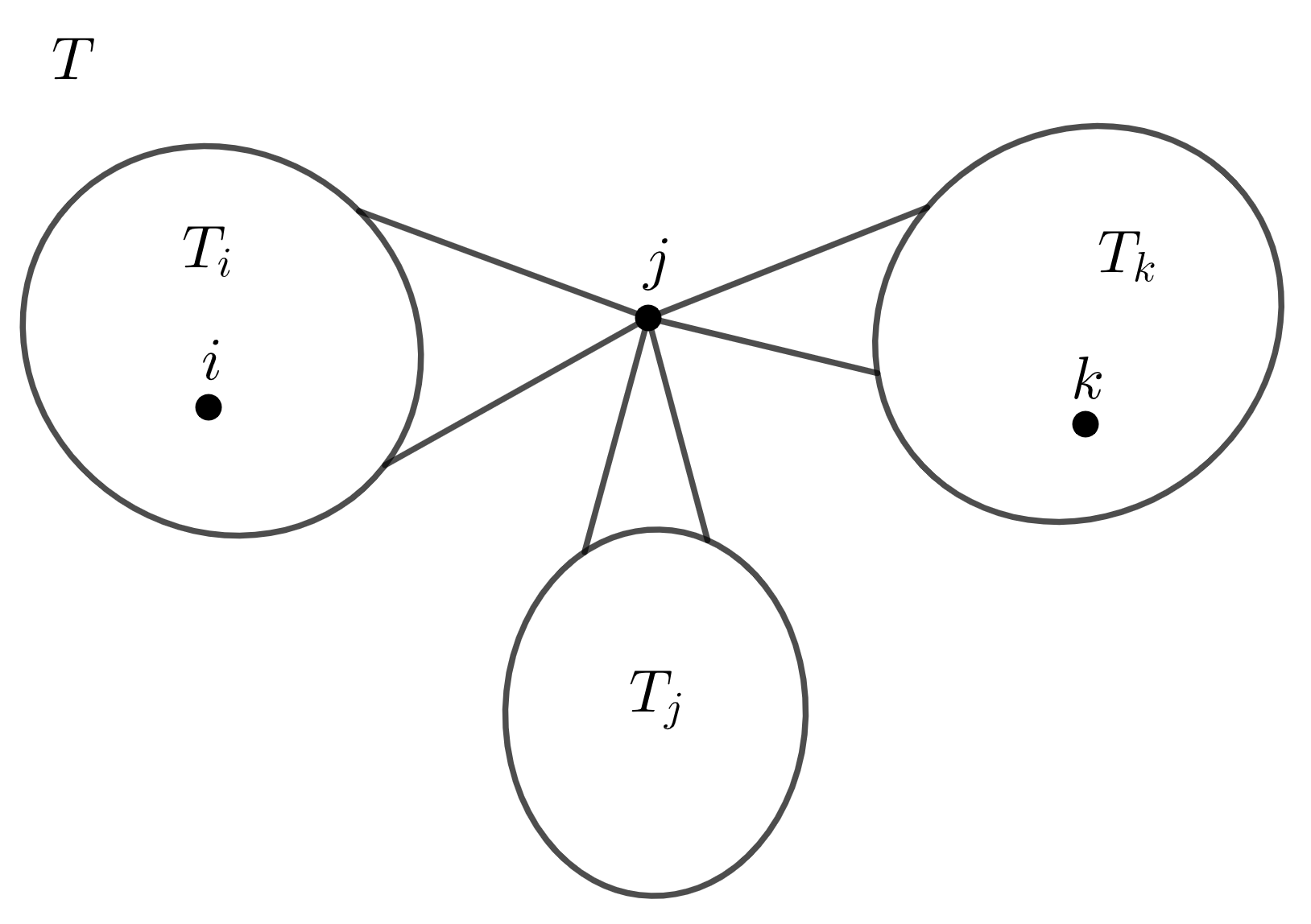}
    \caption{Representation of the tree $T$ with vertices $i$, $j$ and $k$ on the same path.}
    \label{fig:same_path}
\end{figure}

Let $\theta_1(G)$ be the largest eigenvalue of the adjacency matrix of the graph $G$. As $j$ is a cut-vertex, it follows that $\theta_1(T\setminus j)=\max\{\theta_1(T_i),\theta_1(T_j),\theta_1(T_k)\}$. But since $\theta_1(H) < \theta_1(G)$ for every proper subgraph $H$ of a connected graph $G$ (as a consequence of the Perron-Frobenius Theorem, see \cite[Chapter 8]{GodsilRoyle}), we have $\theta_1(T_i) < \theta_1(T \setminus k)$; $\theta_1(T_j) < \theta_1(T \setminus i), \theta_1(T \setminus k)$; $\theta_1(T_k) < \theta_1(T \setminus i)$.

Thus $\theta_1(T \setminus j) < \max\{\theta_1(T \setminus i), \, \theta_1(T \setminus k)\}$, and therefore $j$ cannot be cospectral to both $i$ and $k$.
\end{proof}

Next, we state a main technical lemma, which we will prove in the next section.

\begin{lemma} \label{lem:restriction} Assume vertices $i$, $j$ and $k$ are pairwise strongly cospectral in a graph $G$, and that there exists $v$ so that each $i$, $j$ and $k$ lie in a different component of $G \setminus v$. Thus, for any $\theta \in \Rds$, if $\alpha_i^{G\setminus v}(\theta)=0$, then $\alpha_j^{G\setminus v}(\theta)=\alpha_k^{G\setminus v}(\theta)\neq 0$.
\end{lemma}

The last lemma describes a situation that is not possible.

\begin{theorem}\label{main_result} Let $G$ be a graph with three pairwise cospectral vertices $i$, $j$ and $k$, and assume that there is a cut-vertex $v$ such that these cospectral vertices are in distinct connected components of $G\setminus v$. Then, one of the pairs of cospectral vertices is not strongly cospectral.
\end{theorem}
\begin{proof}
Assume, by contradiction, that the vertices $i$, $j$ and $k$ are pairwise strongly cospectral.

First, note by the Lemma~\ref{lemma:derivative_quotient} that if $\theta$ is a sufficiently large negative number, then $\alpha_i^{G\setminus v}(\theta)$, $\alpha_j^{G\setminus v}(\theta)$ and $\alpha_k^{G\setminus v}(\theta)$ are all negative. On the other hand, if $\theta$ is a sufficiently large positive number, then $\alpha_i^{G\setminus v}(\theta)$, $\alpha_j^{G\setminus v}(\theta)$ and $\alpha_k^{G\setminus v}(\theta)$ are all positive. Let $\tau$ be the smallest real number so that at least one of them is equal to zero, and $\lambda$ the largest real number so that at least one is equal to zero.

Observe also by Lemma~\ref{lemma:derivative_quotient} that $\alpha_i^{G\setminus v}$, $\alpha_j^{G\setminus v}$ and $\alpha_k^{G\setminus v}$ are increasing and continuous in each branch. Finally, note that by Lemma~\ref{lem:restriction} it cannot happen that two terms among $\alpha_i^{G\setminus v}(\theta)$, $\alpha_j^{G\setminus v}(\theta)$ and $\alpha_k^{G\setminus v}(\theta)$ are simultaneously equal to $0$.

Thus, it must be that between $\tau$ and $\lambda$ there is at least one real number $\theta$ for which among $\alpha_i^{G\setminus v}(\theta)$, $\alpha_j^{G\setminus v}(\theta)$ and $\alpha_k^{G\setminus v}(\theta)$ there is a negative number, a positive number and a $0$. But this contradicts Lemma~\ref{lem:restriction}. 

So the vertices $i$, $j$ and $k$ cannot be pairwise strongly cospectral.
\end{proof}

This leads to the promised result.

\begin{corollary}\label{main_result_trees} There is no tree with three pairwise strongly cospectral vertices.
\end{corollary}
\begin{proof}
	By Lemma~\ref{lem:notinpath}, if these vertices exist, then it must be so that there is another vertex of the tree whose removal puts each in a different component. By Theorem~\ref{main_result}, this is not possible.
\end{proof}

It is possible to give a statement analogous to Theorem~\ref{main_result} for matching polynomials, but in this case the definition of strong cospectrality should be given by the analogous statement for matching polynomials of Theorem~\ref{thm:strcospec}. The proof in this case will follow similarly, with the exception that Lemmas~\ref{lem:derivative} and~\ref{lem:wronskian} are replaced by~\cite[Theorem 1.1]{GodsilAlgebraicCombinatorics} and~\cite[Lemma 4.1]{GodsilAlgebraicCombinatorics}, respectively.

In the next section, we work on the proof of Lemma~ \ref{lem:restriction}.

\section{Proof of the key lemma} \label{sec:lemma}

\subsection{Properties of the $\alpha$'s}

We start this section writing the definition of strongly cospectral vertices in terms of the $\alpha$'s defined in \eqref{eq:alpha}.

\begin{lemma}\label{lemma:strongly_cospectral_cf} Let $i$ and $j$ be distinct vertices in the graph $G$. Then, $i$ and $j$ are strongly cospectral if, and only if, $\alpha_i^G=\alpha_j^{G}$ and $\alpha_i^G(\theta)=\alpha_j^{G}(\theta)\neq 0$ whenever $\alpha_i^{G\setminus j}(\theta)=0$ or $\alpha_j^{G\setminus i}(\theta)=0$.
\end{lemma}
\begin{proof} Our proof proceeds via Theorem~\ref{thm:strcospec}. Observe that $\phi^{G\setminus i}=\phi^{G\setminus j}$ if, and only if, $\alpha_i^G=\alpha_j^{G}$. We claim that $\dfrac{\phi^{G\setminus \{i,j\}}}{\phi^{G}}$ has a double pole at $\theta$ if, and only if, $\alpha_i^G(\theta)$, $\alpha_j^{G}(\theta)$, $\alpha_i^{G\setminus j}(\theta)$, $\alpha_j^{G\setminus i}(\theta)$ are all equal to zero. In order to see this, simply notice that,
\[\dfrac{\phi^{G\setminus \{i,j\}}}{\phi^{G}} = \dfrac{1}{\alpha_i^G\alpha_j^{G\setminus i}}=\dfrac{1}{\alpha_j^{G}\alpha_i^{G\setminus j}},\]

\noindent and $\alpha_i^G$, $\alpha_j^{G}$, $\alpha_i^{G\setminus j}$ and $\alpha_j^{G\setminus i}$ have simple zeros by the Lemma~\ref{lemma:derivative_quotient}. 
\end{proof}

In what follows in order to work with the rational functions $\alpha_i^G$ for vertex deleted subgraphs we make use of a technique called \textit{contraction}. This technique is inspired by the theory of continued fractions and has also been used in the context of matching polynomials~\cite{spier2020refined}. 

For distinct vertices $i$ and $j$ in a graph $G$ denote by, 
\[\lambda_{i j}^G := -\left(\dfrac{\sum_{P:i\to j}\phi^{G\setminus P}}{\phi^{G\setminus \{i,j\}}}\right)^2.\]

From here onwards, if $q(x)$ is a rational function, then if write $q(\theta) = \infty$ we simply mean that $\theta$ is a pole of $q(x)$.

Observe that $\lambda_{i j}^G=\lambda_{j i}^G$, and that either $\lambda_{i j}^G(\theta) = \infty$ or $\lambda_{i j}^G(\theta)<0$, for every real number $\theta$. Note also that $\lambda_{i j}^G$ has double zeros and poles. It is an immediate consequence of Lemma~\ref{lem:wronskian} that,

\begin{equation}
	\alpha_i^G=\alpha_i^{G\setminus j}+\dfrac{\lambda_{i j}^G}{\alpha_j^{G\setminus i}}\text{  and } \alpha_j^{G}=\alpha_j^{G\setminus i}+\dfrac{\lambda_{i j}^G}{\alpha_i^{G\setminus j}}. \label{eq:alphaslambdas}
\end{equation}

\noindent In this case we have the following useful observation.

\begin{lemma}\label{lambda_zero} If $\lambda_{i j}^G(\theta)=0$, then $\alpha_i^G(\theta)=\alpha_i^{G\setminus j}(\theta)$. 
\end{lemma}
\begin{proof} Note that $\lambda_{i j}^G$ has double zeros, while, by Lemma~\ref{lemma:derivative_quotient}, $\alpha_j^{G\setminus i}$ has simple zeros. It follows that if $\lambda_{i j}^G(\theta)=0$, then $\dfrac{\lambda_{i j}^G}{\alpha_j^{G\setminus i}}(\theta)=0$, which in turn implies $\alpha_i^G(\theta)=\alpha_i^{G\setminus j}(\theta)$.
\end{proof}

Our next result shows that, more generally, given $\alpha_j^{G\setminus i}(\theta)$, $\alpha_i^{G\setminus j}(\theta)$ and $\lambda_{i j}^G(\theta)\neq \infty$, we can compute $\alpha_i^G(\theta)$ and $\alpha_j^{G}(\theta)$. In the statement of the next result we use the following conventions. For every $C$ in $\Rds \cup\{\infty\}$, $\frac{0}{C}=0$ and $\infty+C=C+\infty=\infty$; if $C\neq \infty$, then $\frac{C}{\infty}=0$; if $C\notin \{0,\infty\}$, then $\frac{C}{0}=\infty$.

\begin{lemma}\label{lambda_finite} If $\lambda_{i j}^G(\theta)\neq \infty$, then, assuming the above conventions,
\[\alpha_i^G(\theta)=\alpha_i^{G\setminus j}(\theta)+\dfrac{\lambda_{i j}^G(\theta)}{\alpha_j^{G\setminus i}(\theta)}.\]
\end{lemma}
\begin{proof} If $\lambda_{i j}^G(\theta)=0$, then, by Lemma~\ref{lambda_zero}, it holds that $\alpha_i^G(\theta)=\alpha_i^{G\setminus j}(\theta)$, and because the zeros of the $\lambda$'s are double and the zeros of the $\alpha$'s are simple, the equality follows.

 Therefore, assume that $\lambda_{i j}^G(\theta)$ is in $(-\infty,0)$. In this case, the result follows immediately, except when $\alpha_i^{G\setminus j}(\theta)=\infty$ and $\alpha_j^{G\setminus i}(\theta)=0$. 

Assume we are in this last situation, then we claim that $\alpha_i^G(\theta)=\infty$. To see this, observe that, by Lemma~\ref{lemma:derivative_quotient}, for every $\varepsilon>0$ sufficiently small it holds that $\alpha_i^{G\setminus j}(\theta-\varepsilon)>0>\alpha_i^{G\setminus j}(\theta+\varepsilon)$ and $\alpha_j^{G\setminus i}(\theta-\epsilon)<0<\alpha_j^{G\setminus i}(\theta+\varepsilon)$.

Then, since $\alpha_i^G=\alpha_i^{G\setminus j}+ \lambda_{i j}^G/\alpha_j^{G\setminus i}$ and $\lambda_{i j}^G$ is negative, we obtain $\alpha_i^G(\theta-\varepsilon)>0>\alpha_i^G(\theta+\varepsilon)$ for every $\epsilon>0$ sufficiently small. It follows by Lemma~\ref{lemma:derivative_quotient} that $\alpha_i^G(\theta)=\infty$.
\end{proof}

In case $\lambda_{i j}^G(\theta)=\infty$ we can still obtain some information.

\begin{lemma}\label{lambda_infinite} If $\lambda_{i j}^G(\theta) = \infty$, then $\alpha_i^{G\setminus j}(\theta)=\alpha_j^{G\setminus i}(\theta)=\infty$. If, in addition, $\alpha_j^{G}(\theta)=\infty$, then $\alpha_i^G(\theta)=\infty$.
\end{lemma}
\begin{proof} First, by \eqref{eq:alphaslambdas}, we have
\[\lambda_{i j}^G=\alpha_i^{G\setminus j}(-\alpha_j^{G\setminus i}+\alpha_i^G)=\alpha_j^{G\setminus i}(-\alpha_i^{G\setminus j}+\alpha_j^{G}).\]

\noindent Now, notice that $\lambda_{i j}^G$ has double poles, while $\alpha_i^G$, $\alpha_j^{G}$, $\alpha_i^{G\setminus j}$ and $\alpha_i^{G\setminus j}$ have simple poles by the Lemma~\ref{lemma:derivative_quotient}. Thus, due to these last expressions for $\lambda_{i j}^G$, it follows that $\lambda_{i j}^G(\theta)=\infty$ can only happen if $\alpha_i^{G \setminus j}( \theta)=\alpha_j^{G\setminus i}(\theta)=\infty$. This proves the first part of the statement.

For the second part of the statement, observe that,
\[\alpha_i^G\alpha_j^{G\setminus i}=\dfrac{\phi^{G}}{\phi^{G\setminus \{i,j\}}}=\alpha_j^{G}\alpha_i^{G\setminus j}.\]

\noindent It follows that if, in addition, $\alpha_j^{G}(\theta)=\infty$, then $\alpha_i^G\alpha_j^{G\setminus i}$ has a double pole at $\theta$, which implies that $\alpha_i^G(\theta)=\infty$, proving the second part of the statement.
\end{proof}

\subsection{The proof of Lemma \ref{lem:restriction}}

\textit{
In this subsection, we assume the graph $G$ has the property that there exists a vertex $v$ such that the vertices $i$, $j$ and $k$ are in distinct connected components of $G\setminus v$.} 

On Figure~\ref{fig_graph} there is a representation of a graph with this property.

\begin{figure}[h]
\begin{center}	\includegraphics[scale=0.2]{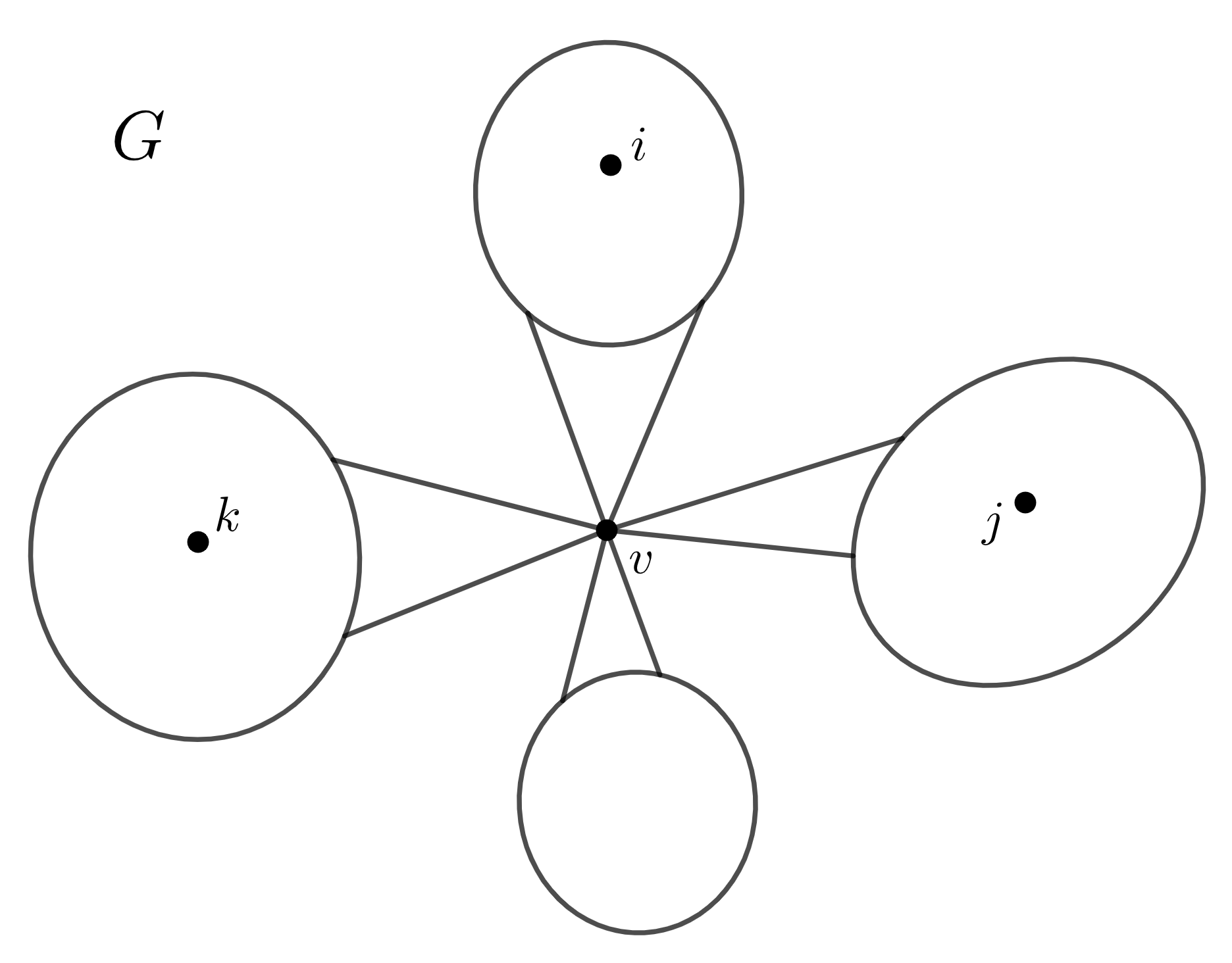}\end{center}
	\caption{A representation of a graph $G$ such that the vertices $i$, $j$ and $k$ are in distinct connected components of $G\setminus v$.}
	\label{fig_graph}
\end{figure}

Our proof of Lemma~\ref{lem:restriction} will follow from conditions imposed by the pairwise strong cospectrality of $i$, $j$ and $k$ on $\alpha_i^{G\setminus v}$, $\alpha_j^{G\setminus v}$ and $\alpha_k^{G\setminus v}$.

Observe that, since $i$, $j$ and $k$ are in distinct connected components of $G\setminus v$, $\alpha_i^{G\setminus v}=\alpha_i^{G\setminus \{v,j\}}=\alpha_i^{G\setminus \{v,j,k\}}$ and $\lambda_{i v}^G=\lambda_{i v}^{G\setminus j}=\lambda_{i v}^{G\setminus \{j,k\}}$, and similar identities by changing the roles of $i$, $j$ and $k$. In what follows, we make heavy use of these facts without further mention.

\begin{lemma}\label{lemma:lambda_equal_zero} If $\lambda_{i v}^G(\theta)=0$, then $\alpha_i^G(\theta)=\alpha_i^{G\setminus v}(\theta)=\alpha_i^{G\setminus j}(\theta)$.
\end{lemma}
\begin{proof} First, note that,

\[\alpha_i^G=\alpha_i^{G\setminus v}+\dfrac{\lambda_{i v}^G}{\alpha_v^{G\setminus i}}\text{ and } \alpha_i^{G\setminus j}=\alpha_i^{G\setminus \{j,v\}}+\dfrac{\lambda_{i v}^G}{\alpha_v^{G\setminus \{j, i\}}}.\]

\noindent As a consequence, by Lemma~\ref{lambda_zero}, $\alpha_i^G(\theta)=\alpha_i^{G\setminus v}(\theta)$ and $\alpha_i^{G\setminus j}(\theta )=\alpha_i^{G\setminus \{j,v\}}(\theta)$, but it also follows that $\alpha_i^{G\setminus \{j,v\}}(\theta)=\alpha_i^{G \setminus v}(\theta)$.
\end{proof}

This last result has the following crucial corollary for strongly cospectral vertices.

\begin{corollary}\label{lambda_not_zero} If vertices $i$ and $j$ are strongly cospectral and $\alpha_i^{G\setminus v}(\theta)=0$, then $\lambda_{i v}^G(\theta)\neq 0$.
\end{corollary}
\begin{proof} Assume otherwise, then, by Lemma~\ref{lemma:lambda_equal_zero}, both $\alpha_i^G(\theta)$ and $\alpha_i^{G\setminus j}(\theta)$ are equal to $\alpha_i^{G\setminus v}(\theta)=0$, which is impossible by Lemma~\ref{lemma:strongly_cospectral_cf}.
\end{proof}

The next results develop the consequences of the conclusion of Corollary~\ref{lambda_not_zero}.

\begin{lemma}\label{forcing_v_in_infty} If $\alpha_i^{G\setminus v}(\theta)=0$ and $\lambda_{i v}^G(\theta)\neq 0$, then, 
\[\alpha_v^{G}(\theta)=\alpha_v^{G\setminus j}(\theta)=\alpha_v^{G\setminus \{j,k\}}(\theta)=\infty.\]
\end{lemma}
\begin{proof} First, note that $\lambda_{i v}^G(\theta)\neq \infty$ because if this was not the case then, by Lemma~\ref{lambda_infinite}, $\alpha_i^{G\setminus v}(\theta)=\infty$, which is impossible.

Then, note that, 

\begin{itemize}
\item[] $\alpha_v^{G}=\alpha_v^{G\setminus i}+ \lambda_{i v}^G/\alpha_i^{G\setminus v}$,

\item[] $\alpha_v^{G\setminus j}=\alpha_v^{G\setminus \{j,i\}}+\lambda_{i v}^G/\alpha_i^{G\setminus \{j,v\}},$
\item[] $\alpha_v^{G\setminus \{j,k\}}=\alpha_v^{G\setminus \{j,k,i\}}+\lambda_{i v}^G/\alpha_i^{G\setminus \{j,k,v\}}.$
\end{itemize}

\noindent But we also have that $\alpha_i^{G\setminus \{j,k,v\}}(\theta)=\alpha_i^{G\setminus \{j,v\}}(\theta)=\alpha_i^{G\setminus v}(\theta)=0$. It follows from Lemma~\ref{lambda_finite} that $\alpha_v^{G}(\theta)=\alpha_v^{G\setminus j}(\theta)=\alpha_v^{G\setminus \{j,k\}}(\theta)=\infty$.
\end{proof}

The next proposition presents the main consequence from the conclusion of Corollary~\ref{lambda_not_zero}, from which more consequences will follow with the hypothesis of strong cospectrality of $j$ and $k$.

\begin{lemma}\label{forcing_equality} If $\alpha_i^{G\setminus v}(\theta)=0$ and $\lambda_{i v}^G(\theta)\neq 0$, then,
\[\alpha_j^{G}(\theta)=\alpha_j^{G\setminus v}(\theta)=\alpha_j^{G\setminus k}(\theta).\]
\end{lemma}
\begin{proof} If $\lambda_{j v}^G(\theta)= \infty$, then Lemma~\ref{lambda_infinite} implies $\alpha_j^{G\setminus v}(\theta)=\infty$. But by Lemma~\ref{forcing_v_in_infty} we also have $\alpha_v^{G}(\theta)=\alpha_v^{G\setminus k}(\theta)=\infty$, from which follows by the second part of Lemma~\ref{lambda_infinite} that $\alpha_j^{G}(\theta)=\alpha_j^{G\setminus k}(\theta)=\infty$. It follows that, $\alpha_j^{G}(\theta)=\alpha_j^{G\setminus v}(\theta)=\alpha_j^{G\setminus k}(\theta)=\infty$, as we wanted.

Now, assume that $\lambda_{j v}^G(\theta)\neq  \infty$. Observe that,
\[\alpha_j^{G}=\alpha_j^{G\setminus v}+\dfrac{\lambda_{j v}^G}{\alpha_v^{G\setminus j}},\quad \alpha_j^{G\setminus k}=\alpha_j^{G\setminus \{k,v\}}+\dfrac{\lambda_{j v}^G}{\alpha_v^{G\setminus \{k, j\}}}.\]

\noindent But by Lemma~\ref{forcing_v_in_infty} we also have  $\alpha_v^{G\setminus j}(\theta)=\alpha_v^{G\setminus \{j,k\}}(\theta)=\infty$. It then follows by Lemma~\ref{lambda_finite} that, $\alpha_j^{G}(\theta)=\alpha_j^{G\setminus v}(\theta)$ and $\alpha_j^{G\setminus k}(\theta)=\alpha_j^{G\setminus \{k,v\}}(\theta)=\alpha_j^{G\setminus v}(\theta)$, which finishes the proof.
\end{proof}

If the vertices $j$ and $k$ are cospectral, this last result has the following corollary.

\begin{corollary}\label{equality_of_alphas} If $\alpha_i^{G\setminus v}(\theta)=0$ and $\lambda_{i v}^G(\theta)\neq 0$, and $j$ and $k$ are cospectral, then, 
\[\alpha_j^{G\setminus k}(\theta)=\alpha_j^{G\setminus v}(\theta)=\alpha_j^{G}(\theta)=\alpha_k^{G}(\theta)=\alpha_k^{G\setminus v}(\theta)=\alpha_k^{G\setminus j}(\theta).\]

\noindent Furthermore, if $j$ and $k$ are strongly cospectral, then this common value is different than zero.
\end{corollary}
\begin{proof} By Lemma~\ref{forcing_equality}, $\alpha_j^{G}(\theta)=\alpha_j^{G\setminus v}(\theta)=\alpha_j^{G\setminus k}(\theta)$ and $\alpha_k^{G}(\theta)=\alpha_k^{G\setminus v}(\theta)=\alpha_k^{G\setminus j}(\theta)$. But $j$ and $k$ are cospectral, so $\alpha_j^{G}(\theta)=\alpha_k^{G}(\theta)$. This proves the first part of the statement. 

Now, if $j$ and $k$ are strongly cospectral, then the common value of these quantities cannot be zero. To see this, observe that if this were not the case, then in particular $\alpha_j^{G}(\theta)=\alpha_j^{G\setminus k}(\theta)=0$, which is impossible by Lemma~\ref{lemma:strongly_cospectral_cf}.
\end{proof}

As a consequence we are ready to prove the key lemma.
\setcounter{lem6}{5}
\begin{lemma*} Assume vertices $i$, $j$ and $k$ are pairwise strongly cospectral in a graph $G$, and that there exists $v$ so that each $i$, $j$ and $k$ lie in a different component of $G \setminus v$. Thus, for any $\theta \in \Rds$, if $\alpha_i^{G\setminus v}(\theta)=0$, then $\alpha_j^{G\setminus v}(\theta)=\alpha_k^{G\setminus v}(\theta)\neq 0$.
\end{lemma*}
\begin{proof} Note that by the Corollary~\ref{lambda_not_zero}, as $i$ and $j$ are strongly cospectral, it follows that $\lambda_{i v}^G(\theta)\neq 0$. But then by the Corollary~\ref{equality_of_alphas}, as $j$ and $k$ are strongly cospectral, it follows that $\alpha_j^{G\setminus v}(\theta)=\alpha_k^{G\setminus v}(\theta)\neq 0$.
\end{proof}

\subsection*{Acknowledgements}

Emanuel Juliano acknowledges the scholarship FAPEMIG/PROBIC. Authors thank Chris Godsil for bringing up the question that motivated this paper.

\bibliographystyle{plain}
\bibliography{references.bib}

	
\end{document}